\documentclass[11pt,letterpaper]{amsart}
\usepackage{amssymb}
\usepackage{mathtools}
\usepackage{amsrefs}
\usepackage{tikz}
\usetikzlibrary{positioning}
\usepackage{fullpage}
\usepackage[indent]{parskip}
\usepackage{enumitem}

\allowdisplaybreaks

\title{Expected value and a Cayley--Menger type formula \\ for the generalized earth mover's distance}

\author{William Q. Erickson}
\address{
William Q.~Erickson\\
Department of Mathematics\\
Baylor University \\ 
One Bear Place \#97328\\
Waco, TX 76798} 
\email{Will\_Erickson@baylor.edu}

\theoremstyle{plain}
\newtheorem{theorem}{Theorem}[section]
\newtheorem{prop}[theorem]{Proposition}
\newtheorem{lemma}[theorem]{Lemma}
\newtheorem{cor}[theorem]{Corollary}

\theoremstyle{definition}
\newtheorem{rem}[theorem]{Remark}
\newtheorem{example}[theorem]{Example}

\newcommand{\EMD}{{\rm EMD}}
\newcommand{\x}{\mathbf{x}}
\newcommand{\wt}{{\rm wt}_{_{\rm L}}\!}

\subjclass[2020]{Primary 60B05; Secondary 49Q22}

\keywords{Earth mover's distance, Wasserstein metric, Cayley--Menger type formulas for hypervolumes}

\begin{document}

\begin{abstract}
   The earth mover's distance (EMD), also known as the 1-Wasserstein metric, measures the minimum amount of work required to transform one probability distribution into another.
   The EMD can be naturally generalized to measure the ``distance'' between any number (say $d$) of distributions.
   In previous work (2021), we found a recursive formula for the expected value of the generalized EMD, assuming the uniform distribution on the standard $n$-simplex.
   This recursion, however, was computationally expensive, requiring $\binom{d+n}{d}$ many iterations.
   The main result of the present paper is a nonrecursive formula for this expected value, expressed as the integral of a certain polynomial of degree at most $dn$.
   As a secondary result, we resolve an unanswered problem by giving a formula for the generalized EMD in terms of pairwise EMDs; this can be viewed as an analogue of the Cayley--Menger determinant formula that gives the hypervolume of a simplex in terms of its edge lengths.
\end{abstract}

\maketitle

\section{Introduction}

\subsection{Generalized earth mover's distance}

The earth mover's distance (EMD) is a metric that is widely used for comparing two probability distributions on a common ground space; roughly speaking, the EMD measures the minimum amount of work required to transform one distribution into the other.
The precise definition of ``work'' (or \emph{cost}) depends on the geometry of the underlying ground space; for this reason, the EMD is often able to capture the similarity of datasets more accurately than other methods of comparison.
Indeed, in recent decades, the EMD (also called the \emph{1-Wasserstein metric}) has been used in a growing range of applications, far beyond its origins in optimal transport theory~\cites{Monge,Villani}, in both the physical and social sciences.
We highlight especially the recent survey~\cite{PZ} treating statistical aspects of the EMD. 

The EMD can naturally be adapted to compare any number (say $d$) of distributions, rather than only two at a time: intuitively, this generalized EMD measures the minimum amount of work required to transform all of the distributions into some common distribution (see Remark~\ref{remark:equalize} below), analogous to the notion of a geometric median, and known as their \emph{1-Wasserstein barycenter}~\cites{AB,PRV}.
Throughout this paper, we use the term \emph{generalized EMD} to describe this ``distance'' between $d$ many distributions.
The general problem also appears in the literature under other names, such as \emph{multi-marginal optimal transport}.
The study of $d$-way contingency tables with given margins can be traced back at least to Fr\'echet~\cite{Frechet} and Feron~\cite{Feron} in the 1950s.
We also point out a detailed algorithmic treatment of the generalized EMD by Bein et al.~\cite{Bein} in the 1990s, motivated by the statistical problem of assigning joint probabilities to samples in a survey; see also the algorithmic paper by Kline~\cite{Kline}.
As a recent example studying the commutative algebraic features of the class of $d$-way contingency tables with given margins, we point out Pistone--Rapallo--Rogantin~\cite{Pistone}.
See~\cite{DKS} for a general historical reference.
We also refer the reader to the many references in~\cite{AB} describing current applications of Wasserstein barycenters in natural language processing and machine learning.

\subsection{Overview of results}

The present paper develops our previous treatment~\cite{Erickson20} of the generalized EMD in two specific directions.
In that treatment, we considered $d$ many points in the standard $n$-simplex $\mathcal{P}_n$, i.e., probability distributions on the ground space $\{1, \ldots, n+1\}$, equipped with the usual $L^1$ distance.
The main result~\cite{Erickson20}*{Thm.~7} was a recursive formula for the expected value of the generalized EMD, assuming the uniform distribution on $\mathcal{P}_n$.
In particular, we presented this formula as a special case of the following:
\[
    \hspace{-8cm}\mathbb{E}[\EMD] \text{ on } \mathcal{P}_{n_1} \times \cdots \times \mathcal{P}_{n_d} =
\]
\begin{equation}
    \label{recursion}
    \frac{\sum_{i=1}^d n_i \big(\mathbb{E}[\EMD] \text{ on } \mathcal{P}_{n_1} \times \cdots \times \mathcal{P}_{n_i-1} \times \cdots \times \mathcal{P}_{n_d}\big) + C(n_1, \ldots, n_d)}{1 + \sum_{i=1}^d n_i},
\end{equation}
where $C$ denotes the $L^1$ dispersion of a sample of $d$ integers.
The desired formula is obtained from the specialization $n_1 = \cdots = n_d = n$.
Note, however, that this specialization depends upon first obtaining $\mathbb{E}[\EMD]$ for every tuple $(n_1, \ldots, n_d)$ such that each $n_i \leq n$.
Although one can exploit symmetry to restrict the recursion to, say, only the weakly increasing sequences $0 \leq n_1 \leq \cdots \leq n_d \leq n$, nevertheless this still requires $\binom{d+n}{d}$ many iterations of the formula~\eqref{recursion} before arriving at the desired expected value; hence for fixed $n$, the length of this recursion is $O(d^n)$, and vice versa.

The primary purpose of this paper is to give a new, nonrecursive formula for the expected value.
Our main result (Theorem~\ref{thm:expected value}) is such a formula:
\begin{equation}
    \label{main result preview}
    \mathbb{E}[\EMD] \text{ on }(\mathcal{P}_n)^d = \int_0^1 \sum_{j=1}^n \sum_{k=1}^{d-1} \min\{k, \: d-k\} \binom{d}{k} F_j(z)^k \Big(1-F_j(z)\Big)^{d-k} \: dz,
\end{equation}
where each $F_j(z)$ is a certain polynomial of degree $n$ (see Lemma~\ref{lemma:Fjz}).
Hence the integrand is itself a polynomial of degree at most $dn$.
In terms of computational complexity, then, this integral formula~\eqref{main result preview} is a vast improvement upon the recursive approach above: for example, when $n=6$ and $d=100$, the recursion~\eqref{recursion} requires approximately $1.7 \times 10^9$ iterations, whereas the integral~\eqref{main result preview} can be evaluated in roughly two seconds using Mathematica on a standard PC (returning the expected value $72.6685$).
We would still like to resolve this formula into a ``closed'' form without the integral, however, and we leave this as an open problem.

The second purpose of this paper is to answer a question we posed in~\cite{Erickson20}*{p.~150}.
In that paper, we observed that when $d=3$, the generalized EMD of three distributions equals half the sum of the three pairwise EMDs.
We also observed that such a phenomenon does \emph{not} continue for $d > 3$: in general, it is impossible to reconstruct the EMD of $d$ distributions from the $\binom{d}{2}$ many pairwise EMDs.
Hence the generalized EMD does indeed contain more information than the ``classical'' EMD (i.e., the $d=2$ case).
Having made these observations, a natural question was to describe the exact relationship between the generalized EMD and the pairwise EMDs: in particular, what additional information is required (beyond the pairwise EMDs) to compute the generalized EMD?  Also, for arbitrary $d$, can we characterize those $d$-tuples of distributions whose EMD \emph{can} be determined from the pairwise EMDs alone, just as in the $d=3$ case?

We settle both these questions in Section~\ref{sec:CM}.
First, we explicitly determine the ``obstruction'' to reconstructing the EMD from the pairwise EMDs: to each $d$-tuple $\x \in (\mathcal{P}_n)^d$, we associate a formal polynomial $G(\x;q)$ whose second derivative at $q=1$ measures this obstruction, resulting in the following general formula (see Theorem~\ref{thm:CM}):
\begin{equation}
    \label{CM preview}
    \EMD(\x) = \frac{1}{d-1}\left[G''(\x;1) + \sum \text{pairwise EMDs}\right].
\end{equation}
Second, we give a statistical interpretation for the vanishing of $G''(\x;1)$, in terms of the cumulative distributions of the distributions in $\x$.
In this way (see Corollary~\ref{cor:EMD=edges}), we characterize those tuples whose EMD truly is just the scaled sum of the pairwise EMDs.
We emphasize that the formula~\eqref{CM preview} can be viewed as an analogue to the celebrated Cayley--Menger determinant formula (see~\eqref{CM} below), which expresses the volume of a simplex in $\mathbb{R}^{d-1}$ in terms of its edge lengths.
Indeed, if one views the generalized EMD as a kind of ``volume'' determined by $d$ distributions, then the pairwise EMDs are precisely the ``edge lengths.''
Of course, the resulting Cayley--Menger type formula~\eqref{CM preview} is quite different than its Euclidean analogue.

We note that the present paper has also provided an opportunity to give shorter and more intuitive proofs of certain results that were fundamental to our original treatment~\cite{Erickson20}; in particular, see Propositions~\ref{old C formula} and~\ref{prop:C is Monge} below.

\subsection{Related work}

The question of computing the expected value of the EMD (in the classical case $d=2$) was first addressed by Bourn and Willenbring~\cite{BW20}, who found a recursive formula by means of generating functions.
Subsequently, Frohmader and Volkmer~\cite{FV} solved this recursion using analytic methods.
The papers~\cites{EK1,EK2} addressed the discrete version of the problem, studying integer-valued histograms rather than probability distributions.

\subsection*{Acknowledgments}
The author sincerely thanks the anonymous referees for their careful reading of the manuscript, and for their valuable suggestions that improved this paper.

\section{Preliminaries}

\subsection{Basic notation}

Throughout the paper, we fix a positive integer $n$, and we study the elements of the standard $n$-simplex
\[
\mathcal{P}_n \coloneqq \Big\{ x = (x_1, \ldots, x_{n+1}) : \textstyle \sum_j x_j = 1, \text{ where each } x_j \geq 0 \Big\} \subset \mathbb{R}^{n+1}.
\]
(To avoid confusion later, we write $\mathcal{P}_n$ rather than the more usual $\Delta^n$, since we will make frequent use of the symbol $\Delta$ to denote successive differences between vector components.)
Each $x \in \mathcal{P}_n$ can be viewed as a probability distribution on the set $[n+1] \coloneqq \{1, \ldots, n+1\}$.
We will denote the cumulative distribution of $x$ by a capital $X$, where $X_j \coloneqq x_1 + \cdots + x_j$.
Since necessarily $X_{n+1} = 1$, we will suppress this final component and simply write $X = (X_1, \ldots, X_n)$.

We fix a positive integer $d \geq 2$, and we use a boldface $\x$ to denote a $d$-tuple of distributions, while a capital boldface $\mathbf{X}$ denotes the $d$-tuple of cumulative distributions.
Within a $d$-tuple, we index the individual distributions by superscripts $1, \ldots, d$.
Hence when there is a superscript only, we are denoting a distribution in $\mathcal{P}_n$; when there is a subscript as well, we are denoting a certain component of that distribution.
As is standard, we use parentheses around indices to denote the order statistics, and we use the symbol $\Delta$ whenever referring to the difference between successive order statistics.
Below in~\eqref{X notation}, for the reader's convenience, we set down all of the notation we will use throughout the paper (see also Example~\ref{ex:G} at the end):
\begin{align}
\label{X notation}
\begin{split}
    x &= (x_1, \ldots, x_{n+1}) \in \mathcal{P}_n,\\
    X &= (X_1, \ldots, X_n), \text{ where } X_j \coloneqq x_1 + \cdots + x_j, \text{ and } X_0 \coloneqq 0,\\
    x_{(i)} & \coloneqq \text{the $i$th smallest component of $x$},\\
    \Delta x_{(i)} & \coloneqq x_{(i+1)} - x_{(i)}, \quad \text{where }\Delta x_{(0)} \coloneqq x_{(1)}, \\
    \x &= (x^1, \ldots, x^d) \in (\mathcal{P}_n)^d,\\
    \mathbf{X} &= (X^1, \ldots, X^d),\\
    x^i_j & \coloneqq \text{the $j$th component of $x^i$},\\
    X^i_j &\coloneqq x^i_1 + \cdots + x^i_j = \text{the $j$th component of $X^i$},\\
    \mathbf{X}^\bullet_j & \coloneqq (X^1_j, \ldots, X^d_j),\\
    X^{(i)}_j & \coloneqq \text{the $i$th smallest component of $\mathbf{X}^\bullet_j$}, \\
    \Delta X^{(i)}_j & \coloneqq X^{(i+1)}_j - X^{(i)}_j.
    \end{split}
\end{align}

\subsection{Dirichlet and beta distributions; order statistics}

As usual in the context of continuous random variables, we write PDF for probability density function, and CDF for cumulative distribution function.
The \emph{Dirichlet distribution}, denoted by ${\rm Dir}(\alpha)$ with parameter $\alpha = (\alpha_1, \ldots, \alpha_{n+1})$, is a certain distribution on  $\mathcal{P}_n$.
For our purposes, we will need only the following standard facts regarding the Dirichlet distribution, the first of which involves the \emph{beta distribution}, denoted by ${\rm Beta}(a,b)$:
\begin{equation}
    \label{beta}
    \text{PDF of ${\rm Beta}(a,b)$ is given by $f(z) = \frac{\Gamma(a+b)}{\Gamma(a) \Gamma(b)} \: z^{a-1} (1-z)^{b-1}$}, \quad z \in [0,1].
\end{equation}
(To avoid confusion with points $x \in \mathcal{P}_n$, we will take $z$ as our variable for all PDFs and CDFs in this paper.)
In general, if $x \sim {\rm Dir}(\alpha)$, then the marginal distributions of ${\rm Dir}(\alpha)$ are the probability distributions of the individual components $x_j$ of $x$, and are given by beta distributions as follows:
\begin{equation}
    \label{lambda i beta}
    x_j \sim {\rm Beta}(\alpha_j, \: \alpha_1 + \cdots + \widehat{\alpha}_j + \cdots + \alpha_{n+1}),
\end{equation}
where the hat denotes an omitted term.
The Dirichlet distribution enjoys the \emph{aggregation property}, meaning that if $x \sim {\rm Dir}(\alpha)$, then removing two components $x_j$ and $x_k$ and replacing them by their sum yields
\[
    (x_1, \ldots, \widehat{x}_j, \ldots, \widehat{x}_k, x_j + x_k, \ldots, x_{n+1}) \sim {\rm Dir}(\alpha_1, \ldots, \widehat{\alpha}_j, \ldots, \widehat{\alpha}_k, \alpha_j + \alpha_k, \ldots, \alpha_{n+1}).
\]
In particular, by repeatedly aggregating initial components, we have the following formula for the aggregation of the first $j$ components:
\begin{equation}
\label{pre Xj}
    (x_1 + \cdots + x_j, x_{j+1}, \ldots, x_{n+1}) \sim {\rm Dir}(\alpha_1 + \cdots + \alpha_j, \alpha_{j+1}, \ldots, \alpha_{n+1}).
\end{equation}
Recalling from~\eqref{X notation} that $X_j \coloneqq x_1 + \cdots + x_j$, we conclude from~\eqref{lambda i beta} and~\eqref{pre Xj} that 
\begin{equation}
    \label{Lambda j Beta}
    X_j \sim {\rm Beta}(\alpha_1 + \cdots + \alpha_j, \: \alpha_{j+1} + \cdots + \alpha_{n+1}).
\end{equation}
In this paper, we will consider only the uniform distribution on $\mathcal{P}_n$, which corresponds to the parameter $\alpha = (1, \ldots, 1)$, also known as the \emph{flat} Dirichlet distribution.
In this case, the equation~\eqref{Lambda j Beta} specializes to
\begin{equation}
    \label{Xj sim Beta flat}
   X_j \sim {\rm Beta}(j, \:n-j+1).
\end{equation}

The CDF of ${\rm Beta}(a,b)$ is called the \emph{regularized incomplete beta function}, typically denoted by $I_z(a,b)$.
In this paper, we will be interested in the CDF of $X_j$, for which we will use the simpler notation $F_j(z)$, rather than $I_z(j, \: n-j+1)$.
We will compute $F_j(z)$ explicitly in Lemma~\ref{lemma:Fjz}.

Finally, we recall the following standard fact regarding order statistics.
Let $Y_1, \ldots, Y_d$ form a random sample from a random variable $Y$, whose CDF is $F_Y(z)$.
Let $Y_{(i)}$ denote the \emph{$i$th order statistic}, i.e., the $i$th smallest element of $\{Y_1, \ldots, Y_d\}$.
It is well known~\cite{Mood}*{Thm.~11, p.~252} that, for each $1 \leq i \leq d$, the $i$th order statistic $Y_{(i)}$ is a random variable whose CDF is given by
\begin{equation}
    \label{F order stat}
    F_{Y_{(i)}}(z) = \sum_{k=i}^d \binom{d}{k} F_Y(z)^k \big(1-F_Y(z)\big)^{d-k}.
\end{equation}

\section{The generalized earth mover's distance}

This section defines the generalized earth mover's distance between a tuple of distributions $\x = (x^1, \ldots, x^d) \in (\mathcal{P}_n)^d$, and culminates in the key Proposition~\ref{prop:EMD equals C}, where we 
prove a straightforward way to compute this distance in terms of the cumulative distributions $X^1, \ldots, X^d$.

\subsection{Classical and generalized EMD}

For the classical case $d=2$, the earth mover's distance (EMD) is typically defined in terms of an optimal solution to the \emph{transportation problem} (often called the Hitchcock, Kantorovich, Koopman, and/or Monge problem).
When $d=2$, the problem is to determine a minimal-cost plan for transporting mass from supply sites to demand sites; both the distribution of the mass among the supply sites
and the desired distribution of the mass among the demand sites are viewed as probability distributions, and the cost of moving one unit of mass between supply site $i$ and demand site $j$ is determined by some cost function $C(i,j)$.
When the cost function is an $L^1$ distance, the minimum required cost is called the EMD between the two distributions.
(More generally, if the cost function is an $L^p$ distance, the minimum cost is called the \emph{$p$-Wasserstein distance}.)
In applications where the number of supply and demand sites is the same, the problem can be viewed as redistributing either distribution in order to obtain the other.
This problem generalizes naturally to any number $d$ of distributions, in which case the problem is to find the minimal-cost way to transform them all into the same distribution (i.e., a 1-Wasserstein barycenter, as mentioned in the introduction).
Explicitly, the generalized transportation problem can be stated as follows.

Let $\x = (x^{1}, \ldots, x^{d}) \in (\mathcal{P}_n)^d$.
Let $C$ be a $d$-dimensional, $(n+1) \times \cdots \times (n+1)$ array, whose entries we write as nonnegative real numbers $C(y)$ for each $y = (y_1, \ldots, y_d) \in [n+1]^d$.
This $C$ (called the \emph{cost array}) depends on the geometry one imposes on the underlying ground space $[n+1]$.
Intuitively, each entry $C(y)$ gives the minimum cost of moving $d$ units of mass (placed at sites $y_1, \ldots, y_d$) to a common site.
The generalized transportation problem asks us to find an array $T_{\x}$ (called a \emph{transport plan} with respect to $\x$), with the same dimensions as $C$, that solves the following linear programming problem:
\begin{align}
    \text{Minimize} \quad & \sum_{\mathclap{y \in [n+1]^d}} C(y) T_{\x}(y), \label{transportation problem}\\
    \text{subject to} \quad & T_{\x}(y) \geq 0 & \text{for all } y \in [n+1]^d, \nonumber\\
    \text{and} \quad & \sum_{\mathclap{y: \: y_i = j}} T_{\x}(y) = x^{i}_j & \text{for all } i \in [d] \text{ and } j \in [n+1]. \nonumber
\end{align}
In the special case where the cost array $C$ is given by an $L^1$ distance (to be defined below in~\eqref{C actual}), we define the \emph{(generalized) earth mover's distance}, still denoted by $\EMD(\x)$, to be the value of the objective function $\sum_y C(y) T^*_\x(y)$, where $T^*_\x$ is an \emph{optimal} transport plan satsifying~\eqref{transportation problem}.

\subsection{Specializing the cost array to one-dimensional ground space}

In this paper, we view the ground space $[n+1] \coloneqq \{1, \ldots, n+1\}$ embedded naturally in the real line, equipped with the usual $L^1$ distance.
Thus the cost of moving one unit of mass from $i$ to $j$ is simply $|i-j|$.
Therefore, to determine a formula for the entries $C(y)$ of the cost array, we observe that if we place one unit of mass at each of the (not necessarily distinct) points $y_1, \ldots, y_d \in [n+1]$, then the minimum cost of moving all $d$ units of mass to a common point in $[n+1]$ is given by the $L^1$ dispersion of the sample $y_1, \ldots, y_d$, namely
\begin{equation}
    \label{C integral}
    \min_{a \in [n+1]} \Big( |y_1 - a| + \cdots + |y_d - a| \Big).
\end{equation}
Although (in the context of the cost array $C$) the coordinates $y_i$ lie in $[n+1]$, nonetheless it will be convenient later for us to extend the domain of $C$ to all of $\mathbb{R}^d$; therefore we make the following definition:
\begin{equation}
    \label{C actual}
    C(y) \coloneqq \min_{a \in \mathbb{R}} \Big( |y_1 - a| + \cdots + |y_d - a| \Big).
\end{equation}
It will follow from the proof of Proposition~\ref{old C formula} that~\eqref{C integral} is precisely the restriction of~\eqref{C actual} to the domain $[n+1]^d$.
(The following is essentially Proposition 1 from our earlier paper~\cite{Erickson20}, but the proof below is shorter and more intuitive.)

\begin{prop}
\label{old C formula}
    Let $C(y)$ be as defined in~\eqref{C actual}.
    For all $y = (y_1, \ldots, y_d) \in \mathbb{R}^d$, we have
    \[
    C(y) = \sum_{i=1}^{\lfloor d/2 \rfloor} y_{(d-i+1)} - y_{(i)},
    \]
    where $y_{(i)}$ denotes the $i$th smallest component of $y$.
\end{prop}

\begin{proof}
    Let $m \coloneqq \lfloor(d+1)/2\rfloor$.
    We claim that the minimum in~\eqref{C actual} is achieved when $a = y_{(m)}$.
    At this value, we have
    \begin{align}
        &\phantom{=.}|y_1 - y_{(m)}| + \cdots + |y_d - y_{(m)}| \nonumber \\
        & = |y_{(1)} - y_{(m)}| + \cdots + |y_{(d)} - y_{(m)}| \nonumber \\
        &= \sum_{i=1}^{m} \big(|y_{(i)} - y_{(m)}| + |y_{(d-i+1)} - y_{(m)}| \big) \label{sorted sum}\\
        &= \sum_{i=1}^m (y_{(m)} - y_{(i)}) + (y_{(d-i+1)} - y_{(m)}) && \text{since $y_{(i)} \leq y_{(m)} \leq y_{(d-i+1)}$} \nonumber\\
        &= \sum_{i=1}^m y_{(d-i+1)} - y_{(i)}, \label{final}
    \end{align}
    which equals the formula stated in the proposition: note that if $d$ is even, then $m = \lfloor d/2 \rfloor$, and if $d$ is odd, then $m = \lfloor d/2 \rfloor + 1$ but the $m$th term in the sum is $y_{(m)} - y_{(m)} = 0$.
    To prove minimality, we use the triangle inequality: namely, for $a \in \mathbb{R}$, we have $|y_{(i)} - a| + |y_{(d-i+1)} - a| \geq y_{(d-i+1)} - y_{(i)}$.
    Hence upon replacing $y_{(m)}$ by an arbitrary $a$, the sum~\eqref{sorted sum} is no less than the sum~\eqref{final}.
\end{proof}

\begin{rem}
\label{remark:equalize}
    As mentioned in the introduction, with respect to the cost function $C(y)$ in~\eqref{C actual}, the quantity $\EMD(\x)$ can be viewed as the minimum cost required to equalize the distributions $x^1, \ldots, x^d$.
    To make this precise, we observe that each transport plan $T_\x$ encodes directions for transforming $x^1, \ldots, x^d$ into some (not necessarily unique) common distribution, as follows.
    For each $y \in [n+1]^d$, choose some $a_y \in [n+1]$ such that $C(y) = \sum_{i=1}^d |y_i - a_y|$.
    (By definition, such an $a_y$ exists.)
    Then move $T_\x(y)$ units of mass from the coordinate $x^i_{y_i}$ to the coordinate $x^i_{a_y}$, for all $1 \leq i \leq d$; in other words, replace each $x^i_{y_i}$ by $x^i_{y_i} - T_\x(y)$, and replace each $x^i_{a_y}$ by $x^i_{a_y} + T_\x(y)$.
    It follows from the two constraints on $T_\x$ in~\eqref{transportation problem} that, upon making these replacements for every position $y$, we have $x^1 = \cdots = x^d$.
    Moreover, recall that for each $y$, we defined $C(y)$ in~\eqref{C actual} so that it would measure the (one-dimensional) cost of moving one unit of mass from each of the $y_i$'s to $a_y$; therefore, the quantity $C(y) T_\x(y)$ does indeed measure the total cost of performing the replacements described above.
    Thus since $\EMD(\x)$ is defined as the minimum value of the objective function $\sum_y C(y) T_\x(y)$ in~\eqref{transportation problem}, we have justified our description of $\EMD(\x)$ as the minimum cost required to equalize the distributions $x^i$.    
\end{rem}

Our next task (Proposition~\ref{prop:C} below) is to restate the formulation of $C(y)$ in Proposition~\ref{old C formula} in several ways which will be useful in different contexts.
Before stating the proposition, we define the following sign function for $1 \leq i \leq d$:
\begin{equation}
    \label{epsilon}
    \epsilon(i) \coloneqq \begin{cases}
        -1, & i < \frac{d+1}{2},\\[1ex]
        \phantom{+}0, & i = \frac{d+1}{2},\\[1ex]
        +1, & i > \frac{d+1}{2}.
    \end{cases}
\end{equation}
Note that $\sum_{i=1}^d \epsilon(i) = 0$.
We also recall the notion of \emph{Lee weight}, which was introduced in~\cite{Lee} in the context of coding theory.
The Lee weight, in a sense, measures an element's ``absolute value'' in the cyclic group $\mathbb{Z}/d\mathbb{Z}$, and is denoted as follows:
\begin{align}
\label{wt}
\begin{split}
    \wt : \{0,1,\ldots, d\} & \longrightarrow \{0,1,\ldots, \lfloor d/2 \rfloor \},\\
    k & \longmapsto \min\{k, \: d-k\}.
\end{split}
\end{align}
We include both $0$ and $d$ in the domain for the sake of convenience; in this paper we have no need for the cyclic group structure, and so we view the domain of $\wt$ as merely a subset of $\mathbb{Z}$.
Comparing~\eqref{epsilon} and~\eqref{wt}, we observe that $\epsilon$ can be viewed as the difference operator on $\wt$, in the sense that $\epsilon(i) = \wt(i-1) - \wt(i)$.
Consequently, we have
\begin{equation}
    \label{wt and epsilon}
    \wt(i) = - \sum_{k=1}^i \epsilon(k).
\end{equation}

\begin{prop}
    \label{prop:C}
    Let $C$ be as defined in~\eqref{C actual}.
    Then $C$ is also given by the two equivalent formulas
    \begin{align}
        C(y) &= \sum_{i=1}^d \epsilon(i) \: y_{(i)} \label{C via epsilon} \\
        &= \sum_{i=1}^{d-1} \wt(i) \cdot \Delta y_{(i)}, \label{C via wt}
    \end{align}
    where $\epsilon(i)$ is defined in~\eqref{epsilon}, $\wt(i)$ is defined in~\eqref{wt}, and all other notation is defined in~\eqref{X notation}.
    Moreover, upon restricting the domain to $[n+1]^d$, we have a third equivalent formula
    \begin{equation}
        \label{C via j}
        C(y) = \sum_{j=1}^{n} \wt\big( \#\{i : y_i > j\}\big), \qquad y \in [n+1]^d.
    \end{equation}
\end{prop}

\begin{proof}
    It is clear that the formula~\eqref{C via epsilon} is just a restatement of Proposition~\ref{old C formula}.
    To show that~\eqref{C via epsilon} is equivalent to~\eqref{C via wt}, we observe that
    \begin{align*}
        \sum_{k=1}^d \epsilon(k) \: y_{(k)} & = \sum_{k=1}^d \epsilon(k) \sum_{i=0}^{k-1} \Delta y_{(i)}\\
        &= \sum_{i=0}^{d-1} \left(\sum_{k=i+1}^d \epsilon(k) \right) \Delta y_{(i)} \\
        &= \sum_{i=0}^{d-1} \underbrace{\left(- \sum_{k=1}^i \epsilon(k) \right)}_{\text{since $\sum_{k=1}^d \epsilon(k) = 0$}} \Delta y_{(i)}\\
        &= \sum_{i=0}^{d-1} \underbrace{\wt(i)}_{\mathclap{\text{by~\eqref{wt and epsilon}}}} \cdot \Delta y_{(i)},
    \end{align*}
    where the $i=0$ term vanishes since $\wt(0)=0$.
    Finally, supposing that $y \in [n+1]^d$, we show that~\eqref{C via wt} is equivalent to~\eqref{C via j} as follows:
    \begin{align*}
        \sum_{i=1}^{d-1} \wt(i) \cdot \Delta y_{(i)} &= \sum_{i=1}^{d-1} \wt(i) \cdot \#\Big\{ j \in [n] : y_{(i)} \leq j < y_{(i+1)} \Big\}\\
        &= \sum_{i=1}^{d-1} \wt(i) \sum_{j=y_{(i)}}^{y_{(i+1)}-1}1\\
        &= \sum_{j=1}^{n} \wt\big(\underbrace{\max\{ i : y_{(i)} \leq j\}}_{\#\{i : y_i \leq j \}} \big) \\
        &= \sum_{j=1}^n \wt \big( \# \{i: y_i > j \} \big),
    \end{align*}
    where the last line follows from the fact that $\wt(k) = \wt(d-k)$.
    \end{proof}

\subsection{An explicit formula for the generalized EMD}

Recall that the generalized EMD is defined as the minimum value of the objective function in~\eqref{transportation problem}, in the special case where the cost array $C$ is the one given in~\eqref{C integral} (or equivalently~\eqref{C actual}).
The takeaway from this subsection will be Proposition~\ref{prop:EMD equals C}, giving an explicit formula for $\EMD(\x)$.
In order to do this, we will appeal to a well-known algorithm for constructing an optimal transport plan $T^*_\x$ to solve the generalized transportation problem~\eqref{transportation problem}.
This algorithm, however, is valid if and only if the cost array $C$ in~\eqref{transportation problem} possesses the \emph{Monge property}, which we now explain below.
It will turn out (see Proposition~\ref{prop:C is Monge}) that the array $C$ given in~\eqref{C integral} does indeed have this Monge property.

Given $y=(y_1, \ldots, y_d)$ and $z = (z_1, \ldots, z_d)$, let $\min(y,z)$ and $\max(y,z)$ denote the vectors which are the componentwise minimum and maximum of $y$ and $z$; that is, $\min(y,z)_i \coloneqq \min\{y_i, z_i\}$ and $\max(y,z)_i \coloneqq \max\{y_i, z_i\}$.
Following~\cite{Bein}*{Def.~2.1}, we say that a $d$-dimensional array $A$ has the \emph{Monge property} if
\begin{equation}
\label{Monge}
    A(\min(y,z))+A(\max(y,z))
\leq A(y)+A(z)
\end{equation}
for all positions $y$ and $z$ in $A$.
It is shown in~\cite{Park}*{Lemma~1.3} that $A$ has the Monge property if and only if every two-dimensional subarray has the Monge property.
For two-dimensional arrays, in turn~\cite{Burkard}*{eqn.~(6)}, it suffices to check adjacent entries only; that is, a two-dimensional array $B$ is Monge if and only if
\[
B(\ell,m) + B(\ell+1, m+1) \leq B(\ell+1,m) + B(\ell, m+1),
\]
for all $\ell$ and $m$ such that all four entries exist.
It follows that a $d$-dimensional array $A$ is Monge if and only if, for all choices of two coordinates, we have
\begin{align}
\label{Monge condition}
\begin{split}
& A(*,\ldots,*,\ell,*,\ldots,*,m,*,\ldots,*) + A(*,\ldots,*,\ell+1,*,\ldots,*,m+1,*,\ldots,*) \\
\leq \: & A(*,\ldots,*,\ell+1,*,\ldots,*,m,*,\ldots,*) + A(*,\ldots,*,\ell,*,\ldots,*,m+1,*,\ldots,*)
\end{split}
\end{align}
for all $\ell$ and $m$ in those two  coordinates (where the *'s denote the remaining coordinates, which are held fixed in all four terms).

\begin{prop}[Prop.~2 in~\cite{Erickson20}]
    \label{prop:C is Monge}
    The cost array $C$, with entries defined as in~\eqref{C integral}, has the Monge property~\eqref{Monge}.
\end{prop}

In~\cite{Erickson20}*{\S8} we resorted to an extensive case-by-case proof of this proposition.
We provide a much shorter proof below, which uses our new Proposition~\ref{prop:C}.

\begin{proof}
    
    We must show that the inequality~\eqref{Monge condition} holds in the special case where $A$ is the cost array $C$ defined in~\eqref{C integral}.
    Let $y = (y_1, \ldots, y_d) \in [n+1]^d$.
    Choose two distinct indices $i<j$ in $[d]$, to play the role of the pair of variable coordinates in~\eqref{Monge condition}.
    Let $y_{_{+i}}$ (resp., $y_{_{+j}}$) denote the vector obtained from $y$ by adding $1$ to the $i$th (resp., $j$th) coordinate; likewise, let $y_{_{+ij}}$ denote the result of adding 1 to both the $i$th and $j$th coordinates.
    Rearranging the inequality~\eqref{Monge condition}, we must show that
    \begin{equation}
        \label{shorter inequality}
        \underbrace{C(y_{_{+i}}) - C(y)}_{\text{LHS}} \geq \underbrace{C(y_{_{+ij}}) - C(y_{_{+j}})}_{\text{RHS}}.
    \end{equation}
    Recalling from~\eqref{C via epsilon} that
    \[
    C(y) = \sum_{i=1}^d \epsilon(i) \: y_{(i)},
    \]
    where $\epsilon$ is the function defined in~\eqref{epsilon}, we observe in~\eqref{shorter inequality} that LHS (resp., RHS) depends only on the relative position of $y_i$ among the order statistics of the coordinates in $y$ (resp., in $y_{_{+j}}$).
    In particular, incrementing $y_i$ by 1 has the following effect on $C(y)$:
    \begin{itemize}
        \item decreases $C(y)$ by 1 if $\max\{k: y_i = y_{(k)} \} < (d+1)/2$;
        \item has no effect on $C(y)$ if $\max\{k: y_i = y_{(k)} \} = (d+1)/2$;
        \item increases $C(y)$ by 1 if $\max\{k: y_i = y_{(k)} \} > (d+1)/2$.
    \end{itemize}
    Thus LHS equals either $-1$, $0$, or $1$, respectively.
    The situation is identical for RHS, where incrementing $(y_{_{+j}})_i = y_i$ by 1 produces the same three effects listed above, upon replacing $y_{(k)}$ by $(y_{_{+j}})_{(k)}$.
    This can all be summarized via the $\epsilon$ function as follows:
    \begin{align}
        \label{LHS RHS}
    \begin{split}
    \text{LHS} &= \epsilon\Big(\max\big\{k: y_i = y_{(k)} \big\} \Big),\\
    \text{RHS} &= \epsilon\Big(\max\big\{k: y_i = (y_{_{+j}})_{(k)} \big\} \Big).
    \end{split}
    \end{align}
    Now, since the relative order of $y_i$ among the coordinates of $y$ must be at least as great as its relative order among the coordinates of $y_{_{+j}}$, we must have 
    \begin{equation}
        \label{arguments}
        \max\big\{k: y_i = y_{(k)} \big\} \geq \max\big\{k: y_i = (y_{_{+j}})_{(k)} \big\}.
    \end{equation}
    Since the function $\epsilon$ is weakly increasing, it follows from~\eqref{LHS RHS} and~\eqref{arguments} that $\text{LHS} \geq \text{RHS}$.
    This establishes~\eqref{shorter inequality} as desired, and completes the proof.
\end{proof}

In terms of writing down an explicit formula for $\EMD(\x)$, the Mongeness of our array $C$ in~\eqref{C integral} is significant for the following reason.
Bein et.~al~\cite{Bein} give an algorithm they call GREEDY$_{\!d}$, to construct a transport plan $T_\x$ for the generalized transportation problem~\eqref{transportation problem}.
(This algorithm is the natural $d$-dimensional generalization of the well-known ``northwest corner rule''~\cite{Hoffman}; see Remark~\ref{rem:NW}.)
They then show that GREEDY$_{\!d}$ produces an optimal transport plan $T^*_\x$ if and only if the cost array $C$ in~\eqref{transportation problem} has the Monge property~\cite{Bein}*{Thm.~4.1}.
We give this algorithm GREEDY$_{\!d}$ below:

\noindent \textbf{Input:} A tuple $\x = (x^1, \ldots, x^d)$ of distributions.\\
\textbf{Output:} An optimal transport plan $T^*_\x$ (assuming that $C$ in~\eqref{transportation problem} has the Monge property).

\begin{enumerate}
\setcounter{enumi}{-1}
    \item \label{step:initialize} Initialize $(y_1, \ldots, y_d) = (1,\ldots,1)$.
    \item \label{step:NW} Set $T^*_\x(y_1, \ldots, y_d)$ equal to $\min\{x^1_{y_1}, \ldots, x^d_{y_d}\}$.
    \item \label{step:modify} Modify one coordinate in each distribution $x^i$ by redefining $x^i_{y_i} \coloneqq x^i_{y_i} - \min \{x^1_{y_1}, \ldots, x^d_{y_d}\}$.
    \item \label{step:zeros} Now there is at least one index $i$ such that $x^i_{y_i} = 0$.
    For each such $i$, set all remaining entries
    \[
   T^*_\x(z) \text{ such that $z_i = y_i$ and all other $z_k \geq y_k$}
    \]
    equal to 0, and increment $y_i$ by 1.
    \item \label{step:repeat} Repeat Steps~\ref{step:NW}--\ref{step:zeros} until all entries of $T^*_\x$ have been filled in.
\end{enumerate}

\begin{rem}
\label{rem:NW}
    The reader may find it helpful to visualize the algorithm above in the original case ($d=2$), where it is known as the ``northwest corner rule.''
    In this case, the $T^*_\x$ constructed by the algorithm is an $(n+1) \times (n+1)$ contingency table whose row and column sums are given by the two distributions $x^1$ and $x^2$, respectively.
    Steps~\ref{step:initialize} and~\ref{step:NW} tell us to begin, quite literally, in the northwest corner of the matrix, and set that entry equal to $\min\{x^1_1, x^2_1\}$.
    After subtracting this value from the first coordinate of both $x^1$ and $x^2$ in Step~\ref{step:modify}, we are directed in Step~\ref{step:zeros} to ``zero out'' the rest of the first row (if $x^1_1$ is now 0) and/or the first column (if $x^2_1$ is now 0).
    We then accordingly move south and/or east by one entry, which takes us back to Step~\ref{step:NW}, namely the northwest corner of the remaining submatrix.
    With each return to Step~\ref{step:NW}, the new northwest corner moves south and/or east, and the remaining submatrix contains one less row and/or column, until we finally reach the southeast corner of the matrix, at which point the algorithm terminates.
    
\end{rem}

It is straightforward to verify that the output $T^*_\x$ of the algorithm GREEDY$_{\!d}$ can be described explicitly as follows.
For each $y \in [n+1]^d$, define the interval
\begin{equation}
    \label{I(y)}
    I(y) \coloneqq \bigcap_{i=1}^d \: [X^i_{y_i-1}, \: X^i_{y_i})
\end{equation}
where for convenience we set $X^i_0 \coloneqq 0$.
Then 
\begin{equation}
    \label{Tx(y) formula}
    T^*_{\x}(y) = \text{length of $I(y)$}.
\end{equation}
To verify~\eqref{Tx(y) formula} in the initial iteration of GREEDY$_{\!d}$ where $y = (1,\ldots,1)$, we have
\begin{align*}
    T^*_\x(1, \ldots, 1) &= \text{length of } [X^1_0, X^1_1) \cap \cdots \cap [X^d_0, X^d_1) \\
    &= \text{length of } [0, x^1_1) \cap \cdots \cap [0, x^d_1)\\
    &= \text{length of } [0, \min\{x^1_1, \ldots, x^d_1\}) \\
    & = \min\{x^1_1, \ldots, x^d_1\},
\end{align*}
which agrees with Step~\ref{step:NW} of the algorithm.
Next, let $i$ be such that $x^i_1 = \min \{x^1_1, \ldots, x^d_1\}$.
Then for every subsequent ${z \neq (1, \ldots, 1)}$ such that $z_i = 1$, there is some $k$ such that $z_k \geq 2$, and thus $[X^i_0, X^i_1) \cap (X^k_{y_k - 1}, X^k_{y_k}) = \varnothing$, since $x^i_1 \leq x^k_1$ implies that $X^i_1 \leq X^k_{1} \leq X^k_{2} \leq \cdots \leq X^k_{n}$.
Hence every subsequent entry $T^*_{\x}(z)$ in the hyperplane $z_i = 1$ is automatically 0, which agrees with Steps~\ref{step:modify}--\ref{step:zeros}.
This argument verifying~\eqref{Tx(y) formula} now holds recursively for each successive $y$.

\begin{prop}
    \label{prop:EMD equals C}
    Let $\x \in (\mathcal{P}_n)^d$, with $\mathbf{X}^\bullet_j$ as in~\eqref{X notation}.
    We have
    \[
    \EMD(\x) = \sum_{j=1}^{n} C(\mathbf{X}^\bullet_j).
    \]
\end{prop}

\begin{proof}
    Since $C$ is Monge by Proposition~\ref{prop:C is Monge}, we have $\EMD(\x) = \sum_y C(y) T^*_{\x}(y)$ where $T^*_{\x}$ is the output of the GREEDY$_{\!d}$ algorithm.
    Hence by~\eqref{Tx(y) formula}, we have
    \begin{equation}
    \label{EMD CI}
    \EMD(\x) = \sum_{\mathclap{y \in [n+1]^d}} C(y) \cdot \text{length of $I(y)$}.
    \end{equation}
Observe from~\eqref{I(y)} that
\begin{equation}
    \label{disjoint union}
    \bigcup_{y \in [n+1]^d} I(y) = [0,1),
\end{equation}
and this union is pairwise disjoint.
Hence for each $t \in [0,1)$, we can set
\[
    y(t) \coloneqq \text{the unique $y$ such that $t \in I(y)$},
\]
whose $i$th component $y_i(t)$ is given by
\begin{equation*}
\label{yi(t)}
    y_i(t) = \#\Big\{ 0 \leq k \leq n : X^i_k \leq t \Big\}.
\end{equation*}
It follows immediately that $y_i(t) > j$ if and only if $X^i_j \leq t$.
Therefore for each $1 \leq j \leq n$, we have 
\begin{equation}
    \label{y's and X's}
    \#\{i: y_i(t) > j\} = \#\{ i : X^i_j \leq t\}.
\end{equation}
To write the right-hand side of~\eqref{y's and X's} as a function of $t \in [0,1)$, it is convenient to use the boxcar function
\[
\Pi_{a,b}(t) \coloneqq \begin{cases}
    1, & t \in [a,b),\\
    0 & \text{otherwise},
\end{cases}
\]
so that we have
\begin{equation}
    \label{boxcar}
    \#\{ i : X^i_j \leq t\} = \sum_{i=1}^d \Pi_{X^i_j, 1}(t) = \sum_{i=1}^{d-1} i \cdot \Pi_{X^{(i)}_j, X^{(i+1)}_j}(t).
\end{equation}
Note on the right-hand side of~\eqref{boxcar} that the supports of the boxcar functions are pairwise disjoint and are contained in $[0,1)$; therefore when composing with some function $\varphi$ and integrating on $[0,1)$, we will have
\begin{equation}
    \label{compose integral}
    \int_0^1 \varphi\left( \#\{i: X^i_j \leq t \} \right) \: dt = 
    \sum_{i=1}^{d-1} \int_{X^{(i)}_j}^{X^{(i+1)}_j} \varphi(i) \: dt.
\end{equation}
Now we can rewrite~\eqref{EMD CI} as
\begin{align*}
    \EMD(\x) &= \sum_{y \in [n+1]^d} \int_{I(y)} C(y) \: dt \\
    &= \int_0^1 C(y(t)) \: dt && \text{by~\eqref{disjoint union}}\\
    &= \int_0^1 \left(\sum_{j=1}^{n} \wt\big(\#\{i: y_i(t) > j\} \big) \right) \: dt && \text{by~\eqref{C via j}}\\
    &= \sum_{j=1}^{n} \int_0^1 \wt\big(\#\{i: X^i_j \leq t \} \big) \: dt && \text{by~\eqref{y's and X's}} \\
    &= \sum_{j=1}^n \sum_{i=1}^{d-1} \int_{X^{(i)}_j}^{X^{(i+1)}_j} \wt(i) \: dt && \text{by~\eqref{compose integral}}\\
    &= \sum_{j=1}^n \sum_{i=1}^{d-1} \wt(i) \cdot \left(X^{(i+1)}_j -X^{(i)}_j \right) && \text{integrating a constant} \\
    &= \sum_{j=1}^n \sum_{i=1}^{d-1} \wt(i) \cdot \Delta X_j^{(i)} && \text{by~\eqref{X notation}}\\
    &= \sum_{j=1}^n C(\mathbf{X}^\bullet_j) && \text{by~\eqref{C via wt}.} \qedhere
\end{align*}
\end{proof}

\section{Main result: expected value}

Before stating our main result, we first obtain an explicit formula for the CDF of the random variable $X_j$ from~\eqref{X notation} and~\eqref{Xj sim Beta flat}.
We record this CDF, which we call $F_j(z)$, in the following lemma.

\begin{lemma}
\label{lemma:Fjz}
    Let $x \in \mathcal{P}_n$ be chosen uniformly at random, and let $F_j(z)$ denote the CDF of the random variable $X_j$.
    We have
    \[
    F_j(z) = \sum_{m=j}^n (-1)^{m-j} \binom{n}{m} \binom{m-1}{j-1} z^m, \qquad \text{for $z \in [0,1]$}.
    \]
\end{lemma}

\begin{proof}
    We have
    \begin{align*}
        F_j(z) & \coloneqq
        \int_0^z \text{PDF of ${\rm Beta}(j, \: n-j+1)$} \: dt && \text{by~\eqref{Xj sim Beta flat}} \\
        &= \int_0^z 
        \frac{\Gamma(n+1)}{\Gamma(j)\Gamma(n-j+1)} \: t^{j-1} (1-t)^{n-j} \: dt && \text{by~\eqref{beta}}\\
        &= \int_0^z \frac{\Gamma(n+1)}{\Gamma(j)\Gamma(n-j+1)} \sum_{\ell = 0}^{n-j} (-1)^\ell \binom{n-j}{\ell} t^{\ell + j - 1} \: dt\\
        &= \sum_{\ell = 0}^{n-j} \frac{\Gamma(n+1)}{\Gamma(j)\Gamma(n-j+1)} (-1)^\ell \binom{n-j}{\ell} \frac{z^{\ell + j}}{\ell+j}.
    \end{align*}
    Since $n$ and $j$ are positive integers, we can replace the gamma functions by factorials; then upon re-indexing the sum via $m \coloneqq \ell + j$, we have
    \begin{align*}
        F_j(z) &= \sum_{m=j}^n (-1)^{m-j} \frac{n!}{(j-1)!(n-j)!} \binom{n-j}{m-j}\frac{z^m}{m}\\
        &= \sum_{m=j}^n (-1)^{m-j} \frac{n!}{(j-1)!(n-j)!} \frac{(n-j)!}{(m-j)!(n-m)!} \frac{(m-1)!}{m!} \: z^m \\
        &= \sum_{m=j}^n (-1)^{m-j} \underbrace{\frac{n!}{(n-m)!} \frac{1}{m!}}_{\binom{n}{m}}\underbrace{\frac{(m-1)!}{1} \frac{1}{(j-1)!(m-j)!}}_{\binom{m-1}{j-1}} \: z^m. \qedhere
    \end{align*}
\end{proof}

Our main result is the following formula for the expected value of the EMD on $(\mathcal{P}_n)^d$.

\begin{theorem}
    \label{thm:expected value}
    Let $\x \in (\mathcal{P}_n)^d$ be chosen uniformly at random.
    The expected value of the generalized EMD is given by
    \[
    \mathbb{E}\Big[ \EMD(\x) \Big] = \int_0^1 \sum_{j=1}^n \sum_{k=1}^{d-1} \wt(k) \binom{d}{k} F_j(z)^k \Big(1-F_j(z)\Big)^{d-k} \: dz,
    \]
    where $\wt(k) \coloneqq \min\{k, \: d-k\}$ is the Lee weight, and $F_j(z)$ is the polynomial given in Lemma~\ref{lemma:Fjz}.
\end{theorem}

\begin{proof}
    We have
    \begin{align}
        \mathbb{E} \Big[\EMD(\x)\Big] &= \mathbb{E} \left[\sum_{j=1}^n C(\mathbf{X}^\bullet_j) \right] && \text{by Proposition~\ref{prop:EMD equals C}} \nonumber \\
        &= \sum_{j=1}^n \mathbb{E}[C(\mathbf{X}^\bullet_j)] \nonumber \\
        &= \sum_{j=1}^n \mathbb{E}\left[ \sum_{i=1}^d \epsilon(i) X^{(i)}_j \right] && \text{by~\eqref{C via epsilon}} \nonumber \\
        &= \sum_{j=1}^n \sum_{i=1}^d \epsilon(i) \cdot \mathbb{E}\left[X^{(i)}_j\right] \label{last line proof}.
    \end{align}
    Since $X^{(i)}_j$ is the $i$th order statistic on $X_j$, it follows from~\eqref{F order stat} that
    \begin{equation}
        \label{Fijz}
        \text{CDF of $X^{(i)}_j$} \eqcolon F_j^{(i)}(z) = \sum_{k=i}^d \binom{d}{k} F_j(z)^k (1-F_j(z))^{d-k},
    \end{equation}
    where $F_j(z)$ is the CDF of the random variable $X_j$, as given in Lemma~\ref{lemma:Fjz}.
    Therefore we obtain
    \begin{align*}
        \mathbb{E}\left[X^{(i)}_j\right] = \int_0^1 1- F^{(i)}_j(z) \: dz 
        &= 1 - \int_0^1 F_j^{(i)}(z) \: dz\\
        &= 1 - \int_0^1 \sum_{k=i}^d \binom{d}{k} F_j(z)^k (1-F_j(x))^{d-k} \; dz,
    \end{align*}
    where the second line follows directly from~\eqref{Fijz}.
    Substituting this in~\eqref{last line proof}, we have
    \begin{align}
    \mathbb{E} \Big[ \EMD(\x) \Big] &= \sum_{j=1}^n \sum_{i=1}^d \epsilon(i) \cdot \left(1 - \int_0^1 \sum_{k=i}^d \binom{d}{k} F_j(z)^k (1-F_j(x))^{d-k} \; dz\right) \nonumber\\
    &= \sum_{j=1}^n \left(\sum_{i=1}^d \epsilon(i) - \int_0^1 \sum_{i=1}^d \sum_{k=i}^d \epsilon(i) \binom{d}{k} F_j(z)^k (1-F_j(x))^{d-k} \; dz \right). \label{last line proof 2}
    \end{align}
    Since we have $\sum_{i=1}^d \epsilon(i) = 0$ by~\eqref{epsilon}, the first sum in the large parentheses vanishes in~\eqref{last line proof 2}; moreover, we have
    \[
    \sum_{i=1}^d \sum_{k=i}^d \epsilon(i) = \sum_{k=1}^d \sum_{i=1}^k \epsilon(i) = - \sum_{k=1}^{d-1} \wt(k),
    \]
    where the second equality follows from~\eqref{wt and epsilon}.
    Applying these two observations to~\eqref{last line proof 2}, we obtain the formula stated in the theorem.
\end{proof}

\begin{rem}
As a sanity check, one can quickly use Theorem~\ref{thm:expected value} to recover the table of expected values we included in our previous paper~\cite{Erickson20}*{Table~2}.
First note, however, that the parameter $n$ in that paper corresponds to $n-1$ in the present paper; also, the table mentioned above gives the \emph{unit normalized} expected values.
Hence, in order to recover Table 2 exactly as it is printed in~\cite{Erickson20}, we must decrease $n$ by 1 before using Theorem~\ref{thm:expected value}, and then divide by $n \lfloor d/2 \rfloor$.
As an example, let $d=10$ and $n=9$; the corresponding entry in~\cite{Erickson20}*{Table~2} is 0.1975 (rounded to four decimal places).
To recover this value, we evaluate our new formula in Theorem~\ref{thm:expected value} at $d=10$ and $n=9-1=8$, which yields approximately 7.9002814.
Upon unit normalizing by dividing by $n \lfloor d/2 \rfloor = 8 \cdot 5 = 40$, we indeed recover $0.1975$.

We should also emphasize that in our original treatment~\cite{Erickson20}, Figure 2 (a plot of expected EMD values for $n=3$ and for $d = 2, \ldots, 100$) had to be generated experimentally, since the runtime of the recursion~\eqref{recursion} was prohibitive at large values of $d$.
Using the formula in Theorem~\ref{thm:expected value}, however, we can now generate the plot with exact values in just a few seconds.
\end{rem}

\section{A Cayley--Menger type formula for the EMD}
\label{sec:CM}

\subsection{EMD viewed as hypervolume}

In this section, we solve a problem we had previously left open regarding the relationship between the generalized EMD of a $d$-tuple $\x = (x^1, \ldots, x^d)$, on one hand, and the EMDs of the individual pairs $(x^i, x^j)$ on the other hand.
In particular, in~\cite{Erickson20}*{Prop.~5} we observed that in the special case $d=3$, the generalized EMD equals half the sum of the pairwise EMDs.
That is, for $x,y,z \in \mathcal{P}_n$, we showed that
\begin{equation}
    \label{EMD CM d=2}
\EMD(x,y,z) = \frac{\EMD(x,y) + \EMD(x,z) + \EMD(y,z)}{2}.
\end{equation}
For $d>3$, however, there was no such formula solely in terms of the pairwise EMDs, and it was unclear what (if anything) could be said in general.

Motivating this problem is its well-known (and ancient) analogue in Euclidean geometry, namely \emph{Heron's formula} for the area of a triangle $\triangle ABC$ in terms of its side lengths $a,b,c$:
\begin{equation}
    \label{Heron}
    \text{area of }\triangle ABC = \sqrt{s(s-a)(s-b)(s-c)},
\end{equation}
where $s = (a+b+c)/2$ is the semiperimeter.
Indeed, one can view~\eqref{EMD CM d=2} as a sort of Heron's formula for the EMD, expressing the ``area'' between three distributions in terms of the three ``side lengths.''
Notice that for the EMD, unlike the Euclidean setting, the formula~\eqref{EMD CM d=2} suggests that this ``area'' equals exactly the ``semiperimeter.''

In the Euclidean setting, a natural question is whether Heron's formula generalizes to a formula for (hyper)volume in higher dimensions;
in other words, can one express the volume of a $(d-1)$-dimensional simplex $\Delta$ solely in terms of its edge lengths (i.e., the volumes of its 1-dimensional faces)?
The answer to this problem is affirmative, made precise by the~\emph{Cayley--Menger determinant} below (see~\cite{Menger}):
\begin{equation}
    \label{CM}
    (\text{volume of } \Delta)^2 = \frac{(-1)^{d}}{2^{d-1} (d-1)!^2} \cdot \det \! \big[\ell^2_{ij}\big]_{i,j=1}^{d+1},
\end{equation}
where $\ell_{ij}$ is the length of the edge between the $i$th and $j$th vertices (and $\ell_{i,d+1} = \ell_{d+1,i} \coloneqq 1$ for all $i \leq d$, and $\ell_{d+1,d+1} \coloneqq 0$).
Hence the volume can be computed from the edge lengths alone.

By analogy, it makes sense to restate the main problem in this section as follows: \emph{find a Cayley--Menger type formula for the EMD.}
By this, we mean a general formula relating the generalized ($d$-fold) EMD to the pairwise EMDs in a $d$-tuple.
In this analogy, we view $\EMD(\x)$ as the volume of the $(d-1)$-dimensional simplex whose $d$ vertices are given by $\x = (x^1, \ldots, x^d)$.
The edges of this simplex are the pairs $(x^k, x^\ell)$ for all $1 \leq k < \ell \leq d$, and so each pairwise $\EMD(x^k, x^\ell)$ is viewed as an ``edge length.''
Note that the EMD of a pair (i.e., the setting where $d=2$) takes a particularly simple form (following directly from Proposition~\ref{prop:EMD equals C}):
\begin{equation}
    \label{EMD d=2}
    \EMD(x^1, x^2) = \sum_{j=1}^n C(X^1_j, X^2_j) = \sum_{j=1}^n | X^1_j - X^2_j|.
\end{equation}

When the EMD is viewed via the geometric analogy described above, the formula~\eqref{EMD CM d=2} states that the area of a triangle is exactly its semiperimeter.
This is the desired Cayley--Menger type formula for the EMD where $d=3$, and the problem now is to generalize this for all values of $d$.

\subsection{A Cayley--Menger type formula for the EMD}

(The reader may find it helpful to look ahead at Example~\ref{ex:G}, which illustrates the results in this final subsection.)
Let $\x \in (\mathcal{P}_n)^d$, and define the following polynomial in the formal indeterminate $q$:
\begin{equation}
    \label{g}
    G(\mathbf{x}; q) \coloneqq \sum_{i=1}^{d-1} \sum_{j=1}^n \Delta X^{(i)}_j \cdot q^{\wt(i)}.
\end{equation}
Note that the derivative of $G(\x;q)$, when evaluated at $q=1$, recovers the EMD:
    \begin{equation}
    \label{G' = EMD}
        G'(\x;1) = \sum_{j=1}^n \underbrace{\sum_{i=1}^{d-1} \wt(i) \cdot \Delta X^{(i)}_j}_{\substack{C(\mathbf{X}^\bullet_j), \\ 
        \text{by~\eqref{C via wt}}}} = \EMD(\x),
    \end{equation}
    where the second equality follows from Proposition~\ref{prop:EMD equals C}.
    Moreover, it is the \emph{second} derivative of $G(\x;q)$ that plays the key role in our following analogue of the Cayley--Menger formula.

\begin{theorem}[Cayley--Menger type formula for EMD]
\label{thm:CM}
    Let $\x = (x^1, \ldots, x^d) \in (\mathcal{P}_n)^d$.
    We have
    \[
    \EMD(\x) = \frac{1}{d-1}\left[G''(\x;1) + \sum_{\mathclap{1 \leq k < \ell \leq d}} \EMD(x^k, x^\ell)\right].
    \]
\end{theorem}

\begin{proof}
    Starting with~\eqref{G' = EMD} and multiplying both sides by $d-1$, we have
    \begin{align}
        (d-1) \cdot \EMD(\x) &= \sum_{j=1}^n \sum_{i=1}^{d-1} \wt(i) \cdot \Delta X^{(i)}_j \cdot (d-1) \nonumber\\
        &= \sum_{j=1}^n \sum_{i=1}^{d-1} \wt(i) \cdot \Delta X^{(i)}_j \cdot \Big( (\wt(i) - 1) + (d-\wt(i))\Big) \nonumber\\
        &= \underbrace{\sum_{j=1}^n \sum_{i=1}^{d-1} \wt(i) \cdot \big( \wt(i) - 1 \big) \cdot \Delta X^{(i)}_j}_{G''(\x; 1)}  + \sum_{j=1}^n \sum_{i=1}^{d-1} \wt(i) \cdot \big( d - \wt(i) \big) \cdot \Delta X^{(i)}_j. \label{two sums}
    \end{align}
    We now rewrite the second sum on the right-hand side of~\eqref{two sums} as follows:
    \begin{align*}
        \sum_{j=1}^n \sum_{i=1}^{d-1} \underbrace{\wt(i) \cdot \big( d - \wt(i) \big)}_{i(d-i)} \cdot \, \Delta X^{(i)}_j &= \sum_{j=1}^n \sum_{i=1}^{d-1} \hspace{-10pt}\underbrace{\sum_{\substack{(k,\ell):\\
        1 \leq k \leq i < \ell \leq d}}}_{\text{$i(d-i)$ many terms}} \hspace{-30pt}\Delta X^{(i)}_j \\
        &= \sum_{j=1}^n \sum_{\substack{(k,\ell): \\ 1 \leq k < \ell \leq d}} \left(\sum_{i = k}^{\ell - 1} \Delta X^{(i)}_j \right)\\
        &= \sum_{\substack{(k,\ell): \\ 1 \leq k < \ell \leq d}} \sum_{j=1}^n \left(X^{(\ell)}_j - X^{(k)}_j\right) \\
        &=\sum_{\substack{(k',\ell'): \\ 1 \leq k' < \ell' \leq d}} \sum_{j=1}^n \left| X^{\ell'}_j - X^{k'}_j\right| \\
        &= \sum_{\substack{(k',\ell'): \\ 1 \leq k' < \ell' \leq d}} \EMD(x^{k'}, x^{\ell'}) && \text{by~\eqref{EMD d=2}.}
    \end{align*}
    Upon substituting this expression into~\eqref{two sums} and dividing both sides by $d-1$, we obtain the desired result.
\end{proof}

Unlike the original Cayley--Menger formula~\eqref{CM} in the Euclidean setting, our EMD analogue in Theorem~\ref{thm:CM} shows that in general, the EMD \emph{cannot} be expressed in terms of the ``edge lengths'' (i.e., pairwise EMDs) alone.
The quantity $G''(\x; 1)$ measures the ``obstruction'' in this regard, namely, how far away the EMD is from a scaled sum of the edge lengths.

In order to gain a more precise understanding of this obstruction measured by $G''(\x;1)$, we revisit our observation~\eqref{G' = EMD} above, which identified $\EMD(\x)$ with $G'(\x;1)$.
The form of~\eqref{G' = EMD} makes it clear that $\EMD(\x) = 0$ if and only if every $\Delta X^{(i)}_j = 0$; this, in turn, occurs if and only if every $x^i$ has the same cumulative distribution $X^i$, meaning that in fact $x^1 = \cdots = x^d$.
This is exactly what we would expect, of course: the EMD measures the cost required to equalize $d$ distributions, and therefore vanishes if and only if those distributions are already all equal.
In fact, as we show below, this first derivative is just a special case (and the strongest case) of a general statistical interpretation of $G(\x;q)$.
Indeed, the higher-order derivatives of $G(\x;q)$ evaluated at $q=1$ also measure, to some extent, the similarity of the cumulative distributions $X^i$.
With each successive derivative, however, this measurement is relaxed more and more, by disregarding the remaining minimum and maximum value (treating them as if they are outliers):

    \begin{prop}
        \label{prop:G derivatives}
        Let $\x \in (\mathcal{P}_n)^d$, with $G(\x;q)$ as defined in~\eqref{g}.
        Let $1 \leq k \leq \lceil d/2 \rceil$.
        The value of $\frac{d^k}{dq^k} G(\x;q) |_{q=1}$ is nonnegative, and it attains zero if and only if
        \[
        X^{(k)}_j = X^{(k+1)}_j = \cdots = X^{(d-k)}_j = X^{(d-k+1)}_j, \qquad \text{for all $1 \leq j \leq n$}.
        \]

        \begin{proof}
            We have
            \[
            \frac{d^k}{dq^k} G(\x;q) = \sum_{i=1}^{d-1} \sum_{j=1}^n \Delta X^{(i)}_j \cdot \underbrace{(\wt(i))(\wt(i)-1) \cdots (\wt(i) - k+1)}_{\text{vanishes if and only if $\wt(i) < k$}} \cdot \, q^{\wt(i) - k},
            \]
            and since the condition $\wt(i) < k$ is equivalent to $i<k$ or $d-i<k$, all nonzero terms must correspond to an index $i$ such that $k \leq i \leq d-k$.
            Thus since every $\Delta X^{(i)}_j \coloneqq X^{(i+1)}_j - X^{(i)}_j$ is nonnegative by definition, the $k$th derivative at $q=1$ is zero if and only if $\Delta X^{(i)}_j = 0$ for all $1 \leq j \leq n$ and for all $i$ such that $k \leq i \leq d-k$.
        \end{proof}
        
    \end{prop}

Combining Theorem~\ref{thm:CM} with the $k=2$ case in Proposition~\ref{prop:G derivatives}, we are now able to characterize the condition under which the EMD actually is just a scaled sum of its edge lengths.

\begin{cor}
\label{cor:EMD=edges}
    We have $\EMD(\x) \geq \frac{1}{d-1} \sum_{k < \ell} \EMD(x^k, x^\ell)$, with equality if and only if
    \[
    X^{(2)}_j = X^{(3)}_j = \cdots = X^{(d-2)}_j = X^{(d-1)}_j \text{ for all $1 \leq j \leq n$.}
    \]
\end{cor}

This now explains the motivating example~\eqref{EMD CM d=2} from our previous paper, where $d=3$ and the EMD is \emph{always} just a scaled sum of edge lengths:
when $d=3$, the chain of equalities in Corollary~\ref{cor:EMD=edges} becomes trivial.

\begin{example}
\label{ex:G}
    This simple example summarizes the various interpretations of the polynomial $G(\x;q)$ described above.
    We take $n=3$ and $d=6$.
    Let $\x = (x^1, \ldots, x^6)$ consist of the distributions below, along with their cumulative distributions $X^i$ (written, as above, without the last component, which is necessarily 1):
    \begin{align*}
        x^1 &= (.2, \: .2, \: .2, \: .4), & X^1 &= (.2, \: .4, \: .6)\\
        x^2 &= (.3, \: .0, \: .4, \: .3), & X^2 &= (.3, \: .3, \: .7)\\
        x^3 &= (.6, \: .0, \: .3, \: .1), & X^3 &= (.6, \: .6, \: .9) \\
        x^4 &= (.0, \: .2, \: .1, \: .7), & X^4 &= (.0, \: .2, \: .3) \\ x^5 &= (.7, \: .1, \: .2, \: .0), & X^5 &= (.7, \: .8, \: 1.0)\\
        x^6 &= (.1, \: .4, \: .0, \: .5), & X^6 &= (.1, \: .5, \: .5).
    \end{align*}
    First we may as well compute $\EMD(\x)$ directly.
    To do this, we recognize that $\mathbf{X}^\bullet_1$, $\mathbf{X}^\bullet_2$, and $\mathbf{X}^\bullet_3$ are the three ``columns'' of coordinates in the cumulative distributions we have displayed above; we compute the cost of each of them using Proposition~\ref{prop:C}:
    \begin{align*}
    \mathbf{X}^\bullet_1 &= (.2, \: .3, \: .6, \: .0, \: .7, \:, .1) \qquad \leadsto \qquad C(\mathbf{X}^\bullet_1) = 1.3,\\
    \mathbf{X}^\bullet_2 &= (.4, \: .3, \: .6, \: .2, \: .8, \:, .5) \qquad \leadsto \qquad C(\mathbf{X}^\bullet_2) = 1.0,\\
    \mathbf{X}^\bullet_3 &= (.6, \: .7, \: .9, \: .3, \: \phantom{.}1, \:, .5) \qquad \leadsto \qquad C(\mathbf{X}^\bullet_1) = 1.2.\\
      \end{align*}
    Thus by Proposition~\ref{prop:EMD equals C}, we have
    \begin{equation}
    \label{EMD in example}
        \EMD(\x) = 1.3 + 1.0 + 1.2 = 3.5.
    \end{equation}
    Next, we will construct the polynomial $G(\x;q)$.
    An intuitive approach is to start with the histograms of $X^1, \ldots, X^6$, shown below:
    
    \begin{center}
    \begin{tikzpicture}[scale=.4]
        \node at (3,0) [label={right,lightgray,scale=.75}:{$0.0$}] {};
        \node at (3,2) [label={right,lightgray,scale=.75}:{$0.2$}] {};
        \node at (3,4) [label={right,lightgray,scale=.75}:{$0.4$}] {};
        \node at (3,6) [label={right,lightgray,scale=.75}:{$0.6$}] {};
        \node at (3,8) [label={right,lightgray,scale=.75}:{$0.8$}] {};
        \node at (3,10) [label={right,lightgray,scale=.75}:{$1.0$}] {};
        \draw[lightgray] (3,10) -- (3,0) -- (0,0);
        
        \draw (0,0) -- ++(0,2)  (1,0) -- ++(0,4)
        (2,0) -- ++(0,6);

        \draw [ultra thick,red!80!black]
        (0,2) -- ++(1,0)
        (1,4) -- ++(1,0)
        (2,6) -- ++(1,0)
        ;

        \node at (1.5, -1) {$X^1$};
    \end{tikzpicture}
    \hfill
    \begin{tikzpicture}[scale=.4]
        \node at (3,0) [label={right,lightgray,scale=.75}:{$0.0$}] {};
        \node at (3,2) [label={right,lightgray,scale=.75}:{$0.2$}] {};
        \node at (3,4) [label={right,lightgray,scale=.75}:{$0.4$}] {};
        \node at (3,6) [label={right,lightgray,scale=.75}:{$0.6$}] {};
        \node at (3,8) [label={right,lightgray,scale=.75}:{$0.8$}] {};
        \node at (3,10) [label={right,lightgray,scale=.75}:{$1.0$}] {};
        \draw[lightgray] (3,10) -- (3,0) -- (0,0);
        
        \draw (0,0) -- ++(0,3)  (1,0) -- ++(0,3)
        (2,0) -- ++(0,7);

        \draw [ultra thick,orange!90!black]
        (0,3) -- ++(1,0)
        (1,3) -- ++(1,0)
        (2,7) -- ++(1,0)
        ;

        \node at (1.5, -1) {$X^2$};
    \end{tikzpicture}\hfill
    \begin{tikzpicture}[scale=.4]
        \node at (3,0) [label={right,lightgray,scale=.75}:{$0.0$}] {};
        \node at (3,2) [label={right,lightgray,scale=.75}:{$0.2$}] {};
        \node at (3,4) [label={right,lightgray,scale=.75}:{$0.4$}] {};
        \node at (3,6) [label={right,lightgray,scale=.75}:{$0.6$}] {};
        \node at (3,8) [label={right,lightgray,scale=.75}:{$0.8$}] {};
        \node at (3,10) [label={right,lightgray,scale=.75}:{$1.0$}] {};
        \draw[lightgray] (3,10) -- (3,0) -- (0,0);
        
        \draw (0,0) -- ++(0,6)  (1,0) -- ++(0,6)
        (2,0) -- ++(0,9);

        \draw [ultra thick,yellow!90!black]
        (0,6) -- ++(1,0)
        (1,6) -- ++(1,0)
        (2,9) -- ++(1,0)
        ;

        \node at (1.5, -1) {$X^3$};
    \end{tikzpicture}\hfill
    \begin{tikzpicture}[scale=.4]
        \node at (3,0) [label={right,lightgray,scale=.75}:{$0.0$}] {};
        \node at (3,2) [label={right,lightgray,scale=.75}:{$0.2$}] {};
        \node at (3,4) [label={right,lightgray,scale=.75}:{$0.4$}] {};
        \node at (3,6) [label={right,lightgray,scale=.75}:{$0.6$}] {};
        \node at (3,8) [label={right,lightgray,scale=.75}:{$0.8$}] {};
        \node at (3,10) [label={right,lightgray,scale=.75}:{$1.0$}] {};
        \draw[lightgray] (3,10) -- (3,0) -- (0,0);
        
        \draw (0,0) -- ++(0,0)  (1,0) -- ++(0,2)
        (2,0) -- ++(0,3);

        \draw [ultra thick,green!70!black]
        (0,0) -- ++(1,0)
        (1,2) -- ++(1,0)
        (2,3) -- ++(1,0)
        ;

        \node at (1.5, -1) {$X^4$};
    \end{tikzpicture}\hfill
    \begin{tikzpicture}[scale=.4]
        \node at (3,0) [label={right,lightgray,scale=.75}:{$0.0$}] {};
        \node at (3,2) [label={right,lightgray,scale=.75}:{$0.2$}] {};
        \node at (3,4) [label={right,lightgray,scale=.75}:{$0.4$}] {};
        \node at (3,6) [label={right,lightgray,scale=.75}:{$0.6$}] {};
        \node at (3,8) [label={right,lightgray,scale=.75}:{$0.8$}] {};
        \node at (3,10) [label={right,lightgray,scale=.75}:{$1.0$}] {};
        \draw[lightgray] (3,10) -- (3,0) -- (0,0);
        
        \draw (0,0) -- ++(0,7)  (1,0) -- ++(0,8)
        (2,0) -- ++(0,10);

        \draw [ultra thick,blue!70!black]
        (0,7) -- ++(1,0)
        (1,8) -- ++(1,0)
        (2,10) -- ++(1,0)
        ;

        \node at (1.5, -1) {$X^5$};
    \end{tikzpicture}\hfill
    \begin{tikzpicture}[scale=.4]
        \node at (3,0) [label={right,lightgray,scale=.75}:{$0.0$}] {};
        \node at (3,2) [label={right,lightgray,scale=.75}:{$0.2$}] {};
        \node at (3,4) [label={right,lightgray,scale=.75}:{$0.4$}] {};
        \node at (3,6) [label={right,lightgray,scale=.75}:{$0.6$}] {};
        \node at (3,8) [label={right,lightgray,scale=.75}:{$0.8$}] {};
        \node at (3,10) [label={right,lightgray,scale=.75}:{$1.0$}] {};
        \draw[lightgray] (3,10) -- (3,0) -- (0,0);
        
        \draw (0,0) -- ++(0,1)  (1,0) -- ++(0,5)
        (2,0) -- ++(0,5);

        \draw [ultra thick,violet]
        (0,1) -- ++(1,0)
        (1,5) -- ++(1,0)
        (2,5) -- ++(1,0)
        ;

        \node at (1.5, -1) {$X^6$};
    \end{tikzpicture}
    \end{center}
    
    \noindent Now if we superimpose all of these histograms on each other, we can visualize each difference $\Delta X^{(i)}_j$ as the distance between the $i$th and $(i+1)$th horizontal bar (counting from the bottom) in the $j$th column (counting from the left); in the picture below, we display $\Delta X^{(4)}_1$ as an example.
    In the space between these two horizontal bars, we write $q^{\wt(i)}$, thus obtaining a full visualization of the terms in the polynomial $G(\x;q)$ as defined in~\eqref{g}: 

    \begin{center}
    \begin{tikzpicture}[scale=.7]
        \node at (3,0) [label={right,gray,scale=.75}:{$0.0$}] {};
        \node at (3,2) [label={right,gray,scale=.75}:{$0.2$}] {};
        \node at (3,4) [label={right,gray,scale=.75}:{$0.4$}] {};
        \node at (3,6) [label={right,gray,scale=.75}:{$0.6$}] {};
        \node at (3,8) [label={right,gray,scale=.75}:{$0.8$}] {};
        \node at (3,10) [label={right,gray,scale=.75}:{$1.0$}] {};
        \draw[gray] (3,10) -- (3,0) -- (0,0);

        \draw [ultra thick,red!80!black]
        (0,2) -- ++(1,0)
        (1,4) -- ++(1,0)
        (2,6) -- ++(1,0)
        ;

        \draw [ultra thick,orange!90!black]
        (0,3) -- ++(1,0)
        (1,3) -- ++(1,0)
        (2,7) -- ++(1,0)
        ;

        \draw [ultra thick,yellow!90!black]
        (0,6) -- ++(1,0)
        (1,6) -- ++(1,0)
        (2,9) -- ++(1,0)
        ;

        \draw [ultra thick,green!70!black]
        (0,0) -- ++(1,0)
        (1,2) -- ++(1,0)
        (2,3) -- ++(1,0)
        ;

        \draw [ultra thick,blue!70!black]
        (0,7) -- ++(1,0)
        (1,8) -- ++(1,0)
        (2,10) -- ++(1,0)
        ;

        \draw [ultra thick,violet]
        (0,1) -- ++(1,0)
        (1,5) -- ++(1,0)
        (2,5) -- ++(1,0)
        ;
        
        \draw (0,0) -- ++(0,7)  (1,0) -- ++(0,8)
        (2,0) -- ++(0,10);

        \node at (.5,.5) {$q$};
        \node at (.5, 1.5) {$q^2$};
        \node at (.5, 2.5) {$q^3$};
        \node at (.5, 4.5) {$q^2$};
        \node at (.5, 6.5) {$q$};
        \node at (1.5, 2.5) {$q$};
        \node at (1.5, 3.5) {$q^2$};
        \node at (1.5, 4.5) {$q^3$};
        \node at (1.5, 5.5) {$q^2$};
        \node at (1.5, 7) {$q$};
        \node at (2.5, 4) {$q$};
        \node at (2.5, 5.5) {$q^2$};
        \node at (2.5, 6.5) {$q^3$};
        \node at (2.5, 8) {$q^2$};
        \node at (2.5, 9.5) {$q$};

        \draw [gray] (-.5,6) -- (-.5,3) (-.75,6) -- (-.25,6)
        (-.75,3) -- (-.25,3);

        \node at (-2,4.5) [text=gray,align=center,scale=.75] {$\Delta X^{(4)}_1 = 0.3$};

        \node at (6,0) {};

        \end{tikzpicture}
    \end{center}

    \noindent Following~\eqref{g}, we scale each term by the vertical distance it occupies in the picture above, then combine like terms to obtain
    \begin{align*}
    G(\x;q) &= (.1+.1+.1+.2+.2+.1)q + (.1+.3+.1+.1+.1+.2)q^2 + (.1+.1+.1)q^3 \\
    &= .8q + .9q^2 + .3q^3.
    \end{align*}
    Now recalling our direct EMD computation in~\eqref{EMD in example}, we can immediately verify the fact~\eqref{G' = EMD}:
    \[
    G'(\x;1) = (.8 + 1.8q + .9q^2)\Big|_{q=1} = 3.5 = \EMD(\x).
    \]
    More importantly, we wish to verify Theorem~\ref{thm:CM}.
    We start by computing
    \begin{equation}
        \label{G'' in example}
        G''(\x;1) = (1.8 + 1.8q) \Big|_{q=1} = 3.6.
    \end{equation}
    The pairwise EMDs are easily computed via~\eqref{EMD d=2}:
    \[
    \begin{array}{llll}
    \EMD(x^1,x^2) = .3, & \EMD(x^1, x^3) = .9, & \EMD(x^1, x^4) = .7, & \EMD(x^1, x^5) = 1.3, \\
    \EMD(x^1, x^6) = .3, & \EMD(x^2, x^3) = .8, & \EMD(x^2, x^4) = .8, & \EMD(x^2, x^5) = 1.2, \\
   \EMD(x^2, x^6) = .6, & \EMD(x^3, x^4) = 1.6, & \EMD(x^3, x^5) = .4, & \EMD(x^3, x^6) = 1.0,\\
    \EMD(x^4, x^5) = 2.0, & \EMD(x^4, x^6) = .6, & \EMD(x^5, x^6) = 1.4. &
    \end{array}
    \]
    The sum of these pairwise EMDs is $13.9$.
    Thus, once again recalling our direct EMD computation in~\eqref{EMD in example}, along with~\eqref{G'' in example}, we verify Theorem~\ref{thm:CM} as follows:
    \[
    \frac{1}{6-1} \Big[ \underbrace{3.6}_{G''(\x;1)} + \underbrace{13.9}_{\substack{\text{sum of }\\ \text{pairwise} \\ \text{EMDs}}}\Big] = \frac{17.5}{5} = 3.5 = \EMD(\x).
    \]
\end{example}

\bibliographystyle{alpha}
\bibliography{references}

\end{document}